\documentclass[11pt]{amsart}
\usepackage{tikz}
\usepackage{booktabs}                                 
\usepackage[english]{babel}
\usepackage{latexsym}
 \usepackage[utf8]{inputenc}
\usepackage{babel}
\usepackage{fancyhdr}
\usepackage{longtable}
\usepackage{mathrsfs}
\usepackage{geometry}
\usepackage{textcomp}
\usepackage{caption}
\usepackage{lmodern}
\usepackage{mathrsfs}
 \usepackage[T1]{fontenc}
\usepackage[all,cmtip]{xy}
\usepackage{epsfig}
 \usepackage{amsmath,amsfonts,amssymb,amsthm}
\usepackage{graphics}
\usepackage{graphicx}
\usepackage{verbatim}
\usepackage{dsfont}
\usepackage{enumerate}  %! used for the library form, but you might find it useful too.

\usepackage{pdflscape}
\usepackage{longtable}
\usepackage{booktabs}
\usepackage{colortbl}
\usepackage[shortlabels]{enumitem}

\newcommand{\C}{\mathbb{C}}

\newcommand{\QQ}{\mathbb{Q}}
\newcommand{\Q}{\mathbb{Q}}
\newcommand{\Z}{\mathbb{Z}}

\newcommand{\mQ}{\mathcal{Q}}

\newcommand{\mR}{\mathcal{R}}

\newcommand{\mH}{\mathcal{H}}

\newcommand{\ord}{\rm{ord}}

\newcommand{\PP}{\mathbb{P}}

\newcommand{\of}{\mathcal{O}}

\newcommand{\F}{\mathcal{F}}

\newcommand{\W}{\bigwedge}

\newcommand{\ii}{\mathrm{i}}

\newcommand{\contr}{\lrcorner}

\DeclareMathOperator{\Ext}{Ext}

\DeclareMathOperator{\Ker}{Ker}

\DeclareMathOperator{\Sym}{Sym}

\DeclareMathOperator{\Gr}{Gr}

\DeclareMathOperator{\Hom}{Hom}

\newtheorem{thm}{Theorem}[section]
\newtheorem{corollary}[thm]{Corollary}

\newtheorem{lemma}[thm]{Lemma}
\newtheorem{proposition}[thm]{Proposition}
\newtheorem{definition}[thm]{Definition}
\newtheorem{rmk}[thm]{Remark}

\newtheorem*{aim*}{Aim of this paper}

\makeatletter    
\def\l@subsection{\@tocline{1}{0,2pt}{2pc}{8mm}{\ \ }} 
\makeatletter    
\def\l@section{\@tocline{1}{0,2pt}{2pc}{8mm}{\ \ }} 

\author{Lev Borisov}
\address{Department of Mathematics\\ Rutgers University\\ Piscataway, NJ 08854}
\email[L.~Borisov]{borisov@math.rutgers.edu}
\author{Enrico Fatighenti }
\address{Department of Mathematical Sciences\\
Loughborough University\\
  LE113TU, UK}
\email[E.~Fatighenti]{e.fatighenti@lboro.ac.uk}

\title[New explicit constructions of surfaces of general type]{New explicit constructions of surfaces of general type}

\begin{document}
\begin{abstract}
We discover a simple construction of a four-dimensional family of smooth surfaces of general type with $p_g(S)=q(S)=0$, $K^2_S=3$ with cyclic fundamental group $C_{14}$. We use a degeneration of the surfaces in this family to find (complicated) explicit equations of six new pairs of fake projective planes. Our methods for finding new fake projective planes involve nontrivial computer calculations which we hope will be applicable in other settings.
\end{abstract}
\maketitle

\section{Introduction}
%{\bf A few words on classification of surfaces and the place of surfaces of general type in this classification.}

Classification of surfaces of general type is one of the most active areas of algebraic geometry. Many examples are known, but a detailed classification is still lacking, and multiple hard problems remain open. 

\medskip
The study of the birational class of a surface $S$ is often reduced to the study of its minimal model. This is especially effective when the Kodaira dimension satisfies $k(S) \geq 0$, since in this case the minimal model is unique. 
Surfaces of general type have maximal (that is, 2) Kodaira dimension. The number of different deformation families is infinite, but still very few examples are known.

\medskip
 To each minimal surface $S$ of general type one associates a triple of numerical invariants, $(p_g, q, K^2_S)$, where $p_g:=h^0(S, K_S)$ and $q:=h^1(S, \of_S)$. These integers determine all other classical numerical invariants, such as $e_{\textrm{top}}(S)=12\chi(\of_S)-K^2_S$ and $P_{m}(S):=h^0(S, mK_S)=\chi(\of_S) + {m \choose 2} K^2_S$. Two widely used ways to produce surfaces of general type are complete intersections of sufficiently high multi-degree or products of curves of genus $g \ge 2$. The resulting surfaces have either large $p_g$ or large $q$. This is a particular manifestation of  a more general phenomenon: producing examples of surfaces of general type with low $p_g$ and $q$ is indeed quite difficult, and a complete classification appears beyond currently available techniques. The most extreme case is that of   $p_g=q=0$.  Surfaces with such invariants are amongst the most famous, since they historically represented counterexamples to the famous Max Noether's conjecture, that stated that any surface $S$ with these prescribed invariants needs to be rational. The  \emph{Bogomolov-Miyaoka-Yau} inequality, that states $ K^2_S \leq 9 \chi(\of_S)$, implies $K^2_S \leq 9$. 

\medskip
The first example of such a surface is due to Godeaux, and is realized as the quotient $Y_5/C_5$, where $Y_5 \subset \PP^3$ is a quintic surface in $\PP^3$ on which the cyclic group $C_5$ acts freely. Surfaces with $p_g=q=0, K^2_S=1$ are therefore called \emph{(numerical) Godeaux surfaces}. Similarly one can construct explicit examples of a surface with $p_g=q=0, K^2_S=2$ as quotients by a $C_7$ action. Surfaces with these prescribed invariants are called \emph{(numerical) Campedelli surfaces}. For $3\leq K^2_S\leq  8$ the situation is in general much less understood. For a recent survey on the surfaces of general type we refer to \cite{bauercatanese}.

\medskip
The extreme case of surfaces with $p_g=q=0$ and $K_S^2=9$ is that of  the famous fake projective planes, i.e. surfaces of general type with Hodge diamond equal to that of $\C\PP^2$. First example of such surface has been given by Mumford \cite{Mumford}. Subsequent work of multiple authors \cite{Klingler, Ishida, Keum06, KK, PY, PY2, cs} culminated in the classification by Cartwright and Steger \cite{cs} which found that there are exactly $50$ complex conjugate pairs of fake projective planes. These planes are computed as free quotients of the two-dimensional complex ball by explicit arithmetic groups.

\medskip Cartwright and Steger  observed that one can obtain many  families of surfaces with $p_g=q=0$ and $K_S^2=3$ as smooth deformations of the quotients of the fake projective planes by a $C_3$ action. In particular, one should find a family of surfaces with $K_S^2=3$ with the fundamental groups $C_{14}$. We have stumbled upon such family in our research and then used it to construct a fake projective plane with first homology $C_{14}$ and symmetry group $C_3 \times C_3$. Afterwards, we found explicit equations of five more (pairs of) fake projective planes in the same class.

\medskip
%One of the results of this paper is an explicit construction  of a family of surfaces $W$ of general type with $p_g=q=0$, $K^2=3$ and cyclic fundamental group  of order 14. In the survey \cite{bauercatanese} the link between such surfaces and fake projective planes (that have $p_g=q=0$ and $K^2=9$) is made explicit. In particular the fundamental group $C_{14}$ is shown to occur by Cartwright and Steger in \cite{cs}, as a smoothing of a 
%quotient of a fake projective plane by an automorphism of order 3. We reverse the process to get (eventually) 

\medskip
{\bf Computer-based approach to constructing specific surfaces of general type.}
Modern software and hardware have made possible breakthroughs in the problems of explicitly constructing surfaces of general type.
In particular, in \cite{BorKeum}, the authors have constructed the first equations of a fake projective plane; in \cite{by} the authors found the equations of a related Cartwright-Steger surface. The results of \cite{BorKeum,by} were subsequently used in \cite{Rito} and \cite{Y19} to fill a gap in the proof of the paper \cite{Y17} on
surfaces with maximum degree of the canonical map.
However, these computational  techniques are still in their infancy, and constructions typically require subtle geometric ideas in order to succeed. One can view this paper as another successful step in developing this emerging field.

\medskip
The paper is organized as follows. We start in Section \ref{sec.2} with an almost classical, but apparently novel, observation that one can construct surfaces with $p_g=q=0, ~K^2=3$ and cyclic fundamental group $C_{14}$ as free quotients of complete intersections of seven special Pl\"ucker hyperplanes in $\Gr(3,V_6)$. 
In Section \ref{sec.3} we follow the remarks of \cite{Keum.comment} to find a quotient of a fake projective plane by $C_3$. In Section \ref{sec.4} we explain the key step that allowed us (with great difficulty) to recover the above fake projective plane. In Section \ref{sec.5} we describe the ensuing construction and the (computer-based) verification of the statement that the surface we found is indeed an FPP. Section \ref{sec.6} explains how we managed to recover five other pairs of FPPs in the same commensurability class. Finally, Section \ref{sec.7} contains a long list of further directions that are naturally inspired by our calculation.

\medskip
{\bf Acknowledgments.} We thank JongHae Keum and Gopal Prasad for interest in this work, and John Cremona for allowing us the use of the Number Theory server at Warwick. L.B. has been partially supported by the National Science Foundation grant DMS-1601907, E.F. has been supported by an EPSRC Doctoral Prize Fellowship based at Loughborough University.

\section{A family of surfaces with $p_g=q=0$ and $K^2=3$}\label{sec.2}
The first construction  of our paper produces  a family of surfaces $S$ of general type with $p_g(S)=q(S)=0$, $K^2_S=3$ and cyclic fundamental group $C_{14}$. This will be done by considering of a family of surfaces $W \subset \PP^{12}$ such that $q(W)=0$, $p_g(W)=13$, $K^2_W=42$ which are equipped with a free action of  $C_{14}$. The surface $S$ will then be realized as the quotient $S=W/C_{14}$. A surface with the same invariants as $W$ appeared in two recent works \cite{by} and \cite{ef}, constructed in different ways. Understanding the connection between these approaches  was the initial motivation behind this project.

\medskip
Let $V_6$ be a complex vector space of dimension $6$ and $V^\vee$ be its dual, with the basis $x_1,\ldots, x_6$. We equip $V^\vee$ with the action of the cyclic group $C_7$ with the generator acting by $x_i\stackrel{\rho}{\mapsto} \varepsilon^i x_i$ where $\varepsilon$ is a primitive $7$-th rooth of $1$.  This action induces a natural action on $\W^3 V_6^\vee$ by $x_i \wedge x_j \wedge x_k\mapsto \varepsilon^{i+j+k}x_i \wedge x_j \wedge x_k$. It is easy to see that $\Gr(3, V_6)$ in its Pl\"ucker embedding is preserved under the induced action on $\PP \W^3 V_6$.

\medskip
The vector space $\W^3 V_6^\vee$ splits into the eigenspaces with respect to the weights for the $C_7$ action. In the table below and for the rest of the paper we denote by $x_{ijk}:=x_i \wedge x_j \wedge x_k$. The basis of each eigenspace is given by the corresponding column of the table.
\begin{center}

\begin{tabular}{@{} *2l @{}*2l @{}*2l @{}*2l @{}*2l @{}*2l @{}*2l }    \toprule
0& &1& &2& &3& &4& &5& &6 \\\midrule
$x_{124}$&& $x_{125}$ && $x_{126}$ && $x_{136}$&& $x_{245}$&& $x_{345}$&& $x_{123}$\\
$x_{356}$&& $x_{134}$ && $x_{135}$ && $x_{145}$&& $x_{146}$&& $x_{156}$&& $x_{346}$\\
 && $x_{456}$ && $x_{234}$ && $x_{235}$&& $x_{236}$&& $x_{246}$&& $x_{256}$

\\\bottomrule
 \hline
\end{tabular}
%\captionof{table}{Possible quotient surfaces and Calabi-Yaus}
\end{center}

\medskip
Consider now a non-degenerate $C_7$-invariant skew form $\omega$  on $V$ given  by $$ \omega= x_1 \wedge x_6 + x_2 \wedge x_5 + x_4\wedge x_3.$$
We will later use that $\omega$ is invariant under the order three map $(x_1,\ldots, x_6)\mapsto (x_2,x_4,x_6,x_1,x_3,x_5)$ which multiplies the subscript by $2$ modulo $7$. The form $\omega$ determines a splitting of $\W^3 V_6^\vee= U_{6} \oplus U_{14}$, where
\begin{equation}\label{u14} U_{6}:= {\rm Image } ( V_6^\vee \stackrel{\wedge \ \omega}{\longrightarrow} \W^3V_6^\vee),~~~
U_{14}:= \Ker (\W^3 V_6^\vee \stackrel{\contr \ \omega^{-1}}{\longrightarrow} V_6^\vee),\end{equation}
with $\contr \ \omega^{-1}$ denoting the contraction of a 3-vector by the inverse symplectic form $\omega^{-1}\in \W^2 V_6$.
Since $\omega$ is $C_7$ invariant, the action of $C_7$ on $\W^3 V_6^\vee$ induces one on $U_{14}$. Explicitly, $U_{14}$ splits into $2$-dimensional eigenspaces with the basis vectors listed in the table below.
\begin{center}

\begin{tabular}{@{} *2l @{}*2l @{}*2l @{}*2l @{}*2l @{}*2l @{}*2l }   \toprule
0& &1& &2& &3& &4& &5& &6 \\\midrule
$x_{124}$&& $x_{125}+x_{134}$ && $x_{126}-x_{234}$ && $x_{136}-x_{235}$&& $x_{146}-x_{245}$&& $x_{156}-x_{345}$&& $x_{256}+x_{346}$\\
$x_{356}$&& $x_{456}$ && $x_{135}$ && $x_{145}$&& $x_{236}$&& $x_{246}$&& $x_{123}$
\\\bottomrule
 \hline
\end{tabular}
 \end{center}
 
 \medskip\noindent
We get our family of $W$ by the following surprisingly simple definition.
\begin{definition} 
Consider the splitting $U_{14} \cong \bigoplus_{i=0}^6 \mH_i $, where $\mH_i$ denotes the eigenspace with respect to the eigenvalue $\varepsilon^i$. For any choice of $H_i \in \mH_i$ we define $W$ as $$W=\Gr(3,V_6) \cap H_0 \cap \ldots \cap H_6.$$ 
\end{definition}

The motivation behind the above definition is the following.
The skew form $\omega$ induces a canonical involution (the \emph{annihilator involution})  acting  as $(-1)$ on $U_6$ and $1$ on $U_{14}$.
It is easy to see by an explicit calculation that this action is induced from the involution on $\Gr(3,V_6)$ which sends a dimension three subspace to its annihilator with respect to $w$. Together with  $C_7$, this involution generates the cyclic group $C_{14}$. \footnote{Note that, unlike  the case of \cite{ef} the involution is not induced from $V_6$: this is eventually the key to the freeness of the action.}  It is clear that  $W$ is $C_{14}$ invariant and we can and will consider the quotient. We are now ready to state the main result of this section.

\begin{thm}
Let $S:= W/ C_{14}$, for a sufficiently general choice of $H_i$. Then $S$ is a smooth surface of general type with $p_g=q=0$, $K^2=3$ and cyclic fundamental group  of order $14$.
\end{thm}

\begin{proof}
We recall that $\Gr(3,V_6)$ is a Fano variety of dimension $9$, index $6$ and degree $42$.  Suppose we can prove that $W$ is a smooth surface with $C_{14}$ acting freely. The adjunction formula implies $K_{W}^2 = 42$, so $K_S^2=3$. We also have $p_g(W)=0$ by the Lefschetz hyperplane theorem, which also implies $p_g(S)=0$. The global sections of the canonical class of $W$ can be identified with $H^0(W,{\mathcal O}(1)) 
\cong \W^3 V_6^\vee / (\C H_0 \oplus \cdots \oplus \C H_6)$, which has dimension $13$. As a result,  $\chi(W,{\mathcal O}_W)=1 + h^{0,2}(W) = 14$, so $\chi(S,{\mathcal O}_S)=1$ which leads to $q(S)=0$.

\medskip
We will now verify the technical statement that $W$ is smooth for general choices of $H_i\in \mH_i$, and moreover that the action of $C_{14}$ on $W$ is free.  This can be easily accomplished by a computer calculation as below, but it would be interesting to find a computer-free argument.

\medskip
Smoothness of $W$, for a sufficiently general choice of coefficients, can be checked in affine coordinate patches on $\PP\W^3 V$. As a first step, we will write the equations of $W$ as 
\begin{align*}
    W=V(&p_0 x_{124} + x_{356}, p_1 x_{456} + x_{125} + x_{134}, p_2  x_{135}  - x_{234} + x_{126}, -p_4 x_{236} + x_{146} + x_{245}, \\
     &-p_3x_{145}-x_{136}+x_{235},
p_5x_{246}-x_{345}+x_{156}, p_6x_{123}+x_{256}+x_{346}) \subset \Gr(3,V_6),
\end{align*}
where $p_i \in \C^*$ are sufficiently general (in particular, not all of them equal). To check the smoothness of $W$ in the affine patch $x_{123}\neq 0$, observe that the corresponding Schubert cell in $\Gr(3,V_6)$ is isomorphic to $\C^9$.
Every point in this cell is the linear span of the rows of the matrix $M$ below.
$$M= \begin{pmatrix} 1 & 0 & 0 & u_1 & u_2 & u_3 \\
0 &1&0 & u_4 &u_5 & u_6 \\
0&0&1& u_7 &u_8&u_9 
\end{pmatrix}$$
In this Schubert cell, the coordinates  $x_{ijk}$ are the $(ijk)$ minors of $M$. The defining of equations of $W$ are then
(nonhomogeneous) equations in $u_i$, with degree up to three. One can then check directly the smoothness of this affine chart of $W$ by computing the Jacobian matrix and its size seven minors. This procedure is then repeated for each of the $20$ coordinate charts.

\medskip
It is easy to show by hand that the action of  $C_7\subset C_{14}$ on $W$ is free.
 Indeed, the fixed points of $C_7$ on $\Gr(3,V_6)$ are the $C_7$-invariant subspaces. Since the weights of $C_7$ action on $V_6$ are all distinct, these are precisely the coordinate subspaces of $V_6$. The corresponding points in $\PP(\W^3V_6)$ have all but one of the $x_{ijk}$ coordinates zero. By direct examination, we see that these points do not lie on a generic $W$.
 
\medskip
It remains to check that the action of the involution in $C_{14}$ is free. To do this, recall that $\iota$ acts as the identiy on $U_{14}$ and as $(-1)$ on $U_6$, the latter being defined in the equation \eqref{u14}. From the explicit presentation of these subspaces given above we can write this action in the coordinate charts and check that the fixed locus of $\iota$ on the ambient $\PP^{12}$ is given by the intersection of $W$ and
%{\bf I don't quite follow the next argument}.
the disjoint union of $\PP^+ \sqcup \PP^-$, these being given by the following equations \begin{align*}
    \PP^+=& V( x_{125}+x_{134}, x_{126}-x_{234}, x_{136}-x_{235},x_{146}-x_{245}, x_{156}-x_{345}, x_{256}+x_{346})\\
    \PP^-=& V( -x_{125}+x_{134}, x_{126}+x_{234}, x_{136}+x_{235},x_{146}+x_{245}, x_{156}+x_{345}, -x_{256}+x_{346}, \\
    & x_{456}, x_{135}, x_{145}, x_{236}, x_{246}, x_{123}, x_{124}, x_{356}).
\end{align*}
Using a computer we can check that, for example if not all $p_i =1$, then the intersection of both $\PP^+$ and $\PP^-$ with $W$ is empty. 
%{\bf Is it really immediate? I think one needs a computer, and there is more to it than just $p_i
%\neq 1$.}
%Similar computation holds for the other involutions $\rho^i \iota$, with identical result.
\end{proof}
\begin{corollary}
The Euler characteristic of $S$ is $9$ and $h^{1,1}(S)=7$.
\end{corollary}
%\begin{rmk}The construction of the involution differs sensibly from the one given in \cite{ef}. The key difference is that the involution here is not induced by any involution on $V_6$. We conjecture that any involution of this type must have a fixed locus, consisting of an elliptic curve and 6 disjoint isolated points.
%
%\end{rmk}

The family of $S$ that we have defined depends on $4$ parameters.  More precisely, we have $7$ dimensional space of choices for the $H_i\in \mH_i$.
However, there is a $(\C^*)^3$ symmetry group of $V_6$
$$
(x_1,\ldots, x_6) \mapsto (\lambda_1 x_1, \lambda_2 x_2, \lambda_3 x_3,  \lambda_3^{-1} x_4,   \lambda_2^{-1} x_5 ,  \lambda_1^{-1} x_6)
$$
 that preserves the action of $C_7$ and the form $w$. 
 This scaling does not change the surface, and therefore we have as a naive moduli count 7-3=4 parameters. Notice that this coincides with the expected number of moduli $M$ of $S$  
$$h^1(S, T_{S}) \geq  \dim M \geq h^1(S, T_{S})-h^2(S, T_{S})=10\chi(\of_{S})-2K_{S}^2=4.$$
There is another less heuristic way of checking that our construction actually gives a four dimensional family. Let $A_W$ denote the affine cone over $W$. The ($\Z$-graded) space $T^1_{A_W} \cong \Ext^1(\Omega^1_{A_W}, \of_{A_W})$ parametrises the deformation of the affine cone, and its degree 0 component represents in particular the deformation of the couple $(W, \of_W(1))$. One can compute, for example using the package "VersalDeformations" of Macaulay2, that this space is 56 dimensional, with its $C_{14}$ invariant subspace being exactly 4 dimensional. However there is no way of making this lengthy computation computer--free.
%{\bf We also need to argue that this construction really gives at least four dimensional family.}

\medskip
To check unobstructedness in the above formula, since $\pi: W \longrightarrow S$ is a finite map, it suffices to check that $H^2(T_W)=0$.  There are several ways of proving this, for example using the Borel--Bott--Weil theorem as follows. Although this is a priori valid only for a general section, it is valid in our case since $\chi(T_W)$ is constant in fibers and $W$ is of general type. Alternatively one can compute the relevant graded component of $T^2_{A_W}$ as above. 
\begin{lemma} Let $W \subset G=\Gr(3,V_6)$ be a smooth complete intersection of seven sections of $\of_{G}(1)^{\oplus 7}$. Then the deformations for  $W$ are unobstructed, i.e. $H^2(T_W)\cong0$.
\end{lemma}
\begin{proof}Let $\F=\of_{G}(1)^{\oplus 7}$. Consider first the tangent sequence for $W$, that is $$0 \to T_W \to T_G|_W \to \F|_W \to 0.$$ Since $\of_X(1)$ is ample $H^1( \F|_W) \cong 0$. Therefore we have $$ 0 \to H^2(T_W) \to H^2(T_G|_W) \to H^2( \F|_W) \to 0.$$
A first observation is that the canonical bundle of $W$ is $\omega_W \cong \of_W(1)$. Therefore $H^2(\of_W(1)) \cong H^2(\omega_W) \cong \C$. This implies $H^2( \F|_W) \cong \C^7$. In order to compute $H^2(T_G|_W) $ we need to use the Koszul complex for $W$, twisted by $T_G$, namely
$$ 0 \to T_G(-7) \to T_G(-6)^7 \to \ldots \to T_G(-2)^{21} \to T_G(-1)^7 \to T_G \to T_G|_W \to 0$$
It is easy to check using Borel--Bott--Weil 
%{\bf need to provide a bit more explanation}
  that $H^i(T_G)=0$ for $i>0$ and that $T_G(-i)$ are acylic for $i=1, \ldots, 5, 7$.  Denote indeed by $\mR$ the tautological bundle of $\Gr(3,V_6)$ and by $\mQ$ its (ample) quotient. Any irreducible homogeneous bundle can therefore be represented in terms of Schur functor as $\Sigma_{\alpha}\mQ \otimes \Sigma_{\beta} \mR$. We denote by $\gamma= (\alpha | \beta)$. Denote by $\delta=(5,4,3,2,1,0)$. By Borel--Bott--Weil theorem any irreducible homogeneous vector bundle will be acylic if $\gamma+\delta$ has repeated entries.
  The tangent bundle $T_G$ to $\Gr(3,V_6)$ is isomorphic to $\Hom(\mR, \mQ)\cong \mR^{\vee} \otimes \mQ$. Therefore the partition associated to it is $(1,0,0,0,0,-1)$, or equivalently $(2,1,1,1,1,0)$ if we use natural dualities. Twisting by $\of_G(-i) \cong (\W^3 \mR)^{\otimes i}$ is equivalent to consider the partition $\gamma(i)=(2,1,1,1+i, 1+i, i)$. It is immediate to check that for $i=1,\ldots, 5, 7$ $\gamma(i)+\delta$ has repeated entries. For $i=0$ there are no repeated entries. The (unique) degree where the bundle has cohomology is therefore identified by the \emph{number of disorder} of the partitions, namely the number of negative differences in the sequence (which is zero in this case). 
  
  The last case is for $i=6$, where we have $H^8 (T_G \otimes \W^6(\of_G(-1)^7)) \cong \C^7$. Indeed, since $\W^9 T_G \cong \of_G(6)$, we get $H^8(G, T_G \otimes \of_G(-6)) \cong H^8 (G, \Omega_G^8) \cong \C$. This implies $ H^2(T_G|_W) \cong \C^7$, and in turn $H^2(T_W) \cong 0$.
\end{proof}

\begin{rmk} We point out that there is another surface with the same invariants as $W$: this is given by a codimension eight linear cut $W'$ of the Grassmannian $\Gr(2,7)$. This surface was already considered in \cite{ef}. The relation between $W$ and $W'$ is still unknown, and we plan to explore it further.
\end{rmk}
We point out that finding an involution that leaves no point fixed on $W'$ could be extremely tricky, if not impossible. Via character theory we can indeed prove that such involution cannot extend to the whole of $\Gr(2,7)$. Indeed it is classically known (see \cite{aut}) that Aut$(\Gr(2,7))= PGL(V_7)$, whereas on $\Gr(3,6)$, PGL$(V_6)$ is a subgroup of index two of the automorphism group. In particular in our construction above our involution was indeed not induced from $V_6$.

\section{A surface with $K^2=3$ and three $A_2$ singularities} \label{sec.3}
In this section we will describe the computer calculations that lead to constructing a quotient of a fake projective plane by a cyclic group $C_3$.  The  Mathematica calculations are available in the  file \cite[Section3.nb]{support}.

\medskip
In the previous section we  constructed a dimension four family of complete intersections in $\Gr(3,V_6)$ which admit a free action of the cyclic group $C_{14}$.  The quotient surfaces have numerical invariants $p_g=q=0,~K^2=3$.
It was remarked in \cite{bauercatanese} that some such surfaces can be constructed as a deformation of a $C_3$ quotient of a certain fake projective plane. So it was only natural to postulate that our family is indeed this family and to look for degenerations of the surfaces that would be $C_3$ quotients of fake projective planes with fundamental group $C_{14}$. 

\medskip
By the classification of Cartwright and Steger, there is one such fake projective plane, up to conjugation. It has a larger symmetry group $C_3\times C_3$, but only one copy of $C_3$ acts trivially on the torsion of its Picard group. This motivated us to look for our complete intersections which have an additional 
condition of having an order $3$ automorphism. We also expected that the quotient has three $A_2$ singularities.
We made an ultimately successful guess to study the two-parameter family of complete intersections in $\Gr(3,V_6)$ given by the equations
$$
0=p_0 x_{124} + x_{356},  0 = x_{456 }+ x_{125}  - x_{143}, 
   0=  x_{135} + x_{243} - x_{216},  0 =  x_{263} + x_{416} - x_{425},
 $$
 $$
  0 =  p_3 x_{415} + x_{316} - x_{325}, p_3 x_{246} + x_{543} - x_{516}, 
  0=  p_3 x_{123} + x_{625} - x_{643}.
$$
Here $p_0$ and $p_3$ are the parameters of the family. We use the convention $x_{ijk} = -x_{jik} = -x_{ikj}$ to make explicit the $C_3$ action that multiplies indices by $2 \hskip -5pt\mod 7$.

\begin{rmk}\label{symrem}
One can observe that the transformation of parameters $(p_0, p_3) \mapsto (p_0^{-1} p_3^3, p_3)$ leads to an isomorphic complete intersection. This symmetry is induced by the coordinate change $x_k \mapsto x_{7-k}$.
\end{rmk}

\medskip
In order to find a surface with $A_2$ singularities, we first tried to determine the locus of singular complete intersections. This was a non-trivial endeavor, since the equations in question were too complicated to be solved directly, at least with our hardware and software. To overcome this difficulty, we first  computed the special case $p_3=1$ to observe that the above complete intersections are singular for  eight values of $p_0$ which are roots of the equation
$$0=999 - 4950 p_0 + 13739 p_0^2 - 23670 p_0^3 + 28532 p_0^4 - 
    23670 p_0^5 + 13739 p_0^6 - 4950 p_0^7 + 999 p_0^8.$$
We then replaced $p_3=1$ by $p_3 = 1 + \epsilon$ for  $\epsilon= 10^{-20}$. We used the gradient descent method built into Mathematica's FindRoot command to calculate the corresponding $8$ roots that deform the above ones. After postulating that there is a polynomial in $(p_0,p_3)$ with reasonably small coefficients that describes the locus of singular complete intersections in $\Gr(3,V_6)$  we found this polynomial to be 
$$
{
\small
\begin{array}{l}
0=
1280 p_0^8 - 6144 p_0^7 p_3 + 4608 p_0^8 p_3 + 10240 p_0^6 p_3^2 - 
 18176 p_0^7 p_3^2 + 7424 p_0^8 p_3^2 - 8192 p_0^5 p_3^3 + 
 41984 p_0^6 p_3^3 
\\[.2em]
 - 30976 p_0^7 p_3^3 + 7040 p_0^8 p_3^3 + 
 4096 p_0^4 p_3^4 - 44032 p_0^5 p_3^4 + 75328 p_0^6 p_3^4 - 
 34560 p_0^7 p_3^4 + 4320 p_0^8 p_3^4 + 28672 p_0^4 p_3^5 - 
 \\[.2em]
 108288 p_0^5 p_3^5 + 92448 p_0^6 p_3^5 - 25568 p_0^7 p_3^5 + 
 1760 p_0^8 p_3^5 - 8192 p_0^3 p_3^6 + 89600 p_0^4 p_3^6 - 
 155264 p_0^5 p_3^6 + 77760 p_0^6 p_3^6 
 \\[.2em]
 - 12784 p_0^7 p_3^6 + 
 464 p_0^8 p_3^6 - 44032 p_0^3 p_3^7 + 169216 p_0^4 p_3^7 - 
 147984 p_0^5 p_3^7 + 46624 p_0^6 p_3^7 - 4320 p_0^7 p_3^7 + 72 p_0^8 p_3^7 
 \\[.2em]
 + 
 10240 p_0^2 p_3^8 - 108288 p_0^3 p_3^8 + 190656 p_0^4 p_3^8 - 
 100696 p_0^5 p_3^8 + 19440 p_0^6 p_3^8 - 968 p_0^7 p_3^8 + 5 p_0^8 p_3^8 + 
 41984 p_0^2 p_3^9 
 \\[.2em]
 - 155264 p_0^3 p_3^9 + 149952 p_0^4 p_3^9 - 
 50348 p_0^5 p_3^9 + 5778 p_0^6 p_3^9 - 142 p_0^7 p_3^9 - 6144 p_0 p_3^{10} + 
 75328 p_0^2 p_3^{10} - 147984 p_0^3 p_3^{10} 
 \\[.2em]
 + 85745 p_0^4 p_3^{10} - 
 18498 p_0^5 p_3^{10} + 1177 p_0^6 p_3^{10} - 12 p_0^7 p_3^{10} - 
 18176 p_0 p_3^{11} + 92448 p_0^2 p_3^{11} - 100696 p_0^3 p_3^{11} 
 \\[.2em]
 + 
 37488 p_0^4 p_3^{11} - 4852 p_0^5 p_3^{11} + 164 p_0^6 p_3^{11} + 1280 p_3^{12} - 
 30976 p_0 p_3^{12} + 77760 p_0^2 p_3^{12} - 50348 p_0^3 p_3^{12} 
 \\[.2em]
 + 
 11916 p_0^4 p_3^{12} - 846 p_0^5 p_3^{12} + 10 p_0^6 p_3^{12} + 4608 p_3^{13} - 
 34560 p_0 p_3^{13} + 46624 p_0^2 p_3^{13} - 18498 p_0^3 p_3^{13} + 
 2644 p_0^4 p_3^{13} 
 \\[.2em]
 - 86 p_0^5 p_3^{14} + 7424 p_3^{14} - 25568 p_0 p_3^{14} + 
 19440 p_0^2 p_3^{14} - 4852 p_0^3 p_3^{14} + 350 p_0^4 p_3^{14} - 4 p_0^5 p_3^{15} 
 + 
 7040 p_3^{15} 
 \\[.2em]
 -
 12784 p_0 p_3^{15} + 5778 p_0^2 p_3^{15} - 846 p_0^3 p_3^{15} + 
 28 p_0^4 p_3^{16} + 4320 p_3^{16} - 4320 p_0 p_3^{16} + 1177 p_0^2 p_3^{16} - 
 86 p_0^3 p_3^{16} + p_0^4 p_3^{17} 
 \\[.2em]
 + 1760 p_3^{17} - 968 p_0 p_3^{17} + 
 164 p_0^2 p_3^{17} - 4 p_0^3 p_3^{18} + 464 p_3^{18} - 142 p_0 p_3^{18} + 
 10 p_0^2 p_3^{19} + 72 p_3^{19} - 12 p_0 p_3^{19} + 5 p_3^{20}.
\end{array}
}
$$
After finding the above locus of singular complete intersections, we made an educated guess to look for singularities of the above curve in hopes of finding complete intersections with singular points of type $A_2$. This was a straightforward computer calculation that lead to twelve singular points, which we then looked at in detail. Up to complex conjugation and symmetry of Remark \ref{symrem}, there was one solution defined over the expected field, with the expected  $42$ singularities of type $A_2$, which form one orbit of the semidirect product of $C_{14}$ and  $C_3$.
%{\bf: 42 is not totally a random number in this context. In particular $W$ looks invariant w.r.t. the group of order 42 generated by $C_7$ and $i \mapsto 2i$ (mod 7). In the $\Gr(2,7)$ model this $G_{42}$ was the biggest automorphism group under which the whole family was invariant, maybe it's the same with this set of equations as well}
Specifically, we got
$$
p_0 =   \frac 1 {256}(-2475 + 49  i \sqrt 7 - 35 i \sqrt{15} + 231 \sqrt{105}), ~
  p_3 = \frac 1{16} (-17 + 7 i \sqrt{15}).
$$
In what follows we will denote this surface by $W_{\Gr}$.
It is worth mentioning that the ratios of the Pl\"ucker coordinates  $x_{ijk}$ of the $42$ singular points of $W_{\Gr}$ are not at all pleasant. Some of them are roots of equations of degree $168$ with coefficients that are tens of digits long. Nonetheless, we were able to verify by computer that these are indeed  $A_2$ singularities.

\medskip
We will now focus our attention on the quotient surface $W_{\Gr}/C_{14} $ which is a singular surface with $K^2=3$ and three $A_2$ singularities. Such surfaces have been shown in \cite{Keum.comment} to be  quotients of  fake projective planes by a cyclic group $C_3$. For our purposes, it will be convenient to work with the unramified double cover $X=W_{\Gr}/C_{7} $  of $W_{\Gr}/C_{14} $ which is a surface with $K^2=6$ and six singular points of type $A_2$. We found equations of this surface by the following method, using Mathematica, analogous to \cite{BorKeum}. It should also be possible to use symbolic rather than numerical methods for it, or at the very least check symbolically that the invariants of $C_7$ action satisfy these equations. Notice that a smoothing of this surface gives us equations of the surface constructed in \cite{ef}.

\medskip
\begin{enumerate}
\item Construct multiple points on the surface $W_{\Gr}$, numerically with several hundred digit accuracy. The points are generated randomly.

\item Construct sections of $2K_X$ as quadratic polynomials in Pl\"ucker coordinates on $W_{\Gr}$ which are invariant under $C_7$. There is an $8$-dimensional space of these sections. We initially picked the basis
$$
\begin{array}{l}
(U_0, \ldots ,U_7) = \Big( \frac 14 (-8925 - 1561 i \sqrt 7+ 1115 i \sqrt{15}
        + 833 \sqrt{105} )\, x_{124}^2,  32 (7 - i \sqrt{15})\, x_{456}  x_{123}, 
 \\[.2em]
 32 (7 - i \sqrt{15}) \,x_{135} x_{246}, 
  32 (7 - i \sqrt{15}) \, x_{263} x_{415},
  128\, x_{456} (x_{625} + x_{643}),   128\, x_{135} (x_{543} + x_{516}), 
  \\[.2em]
  128\, x_{263} (x_{316} + x_{325}),
 \frac 13((136 - 56 i \sqrt{15}) (x_{123} (x_{125} + x_{143}) + 
          x_{246} (x_{243} + x_{216}) + x_{415} (x_{416} + x_{425})) 
  \\[.2em]
          + 
  128( x_{456} (x_{625} + x_{643}) + x_{135} (x_{543} + x_{516}) + 
          x_{263} (x_{316} + x_{325})))
          \Big)
\end{array}
$$
in order to simplify the resulting degree two equations.

\item Use the aforementioned points of $W_{\Gr}$ to find a basis of equations of degree $2$ and $3$ among the sections of $K_X$, numerically.

\item Identify the coefficients of these equations with algebraic numbers. These numbers are guaranteed to be in the field $\QQ(\sqrt{-15},\sqrt{-7})$, which provides a good check of the calculations. The equations of $X$  in terms of $U_j$  are presented in the accompanying file \cite[EquationsOfXintermsofU]{support}.

\item Calculate the Hilbert polynomial of the scheme cut out by these equations to ensure that they cut out the surface $X$ scheme-theoretically.

\end{enumerate}

\begin{rmk}
It is currently a bit of an art to find the generators of the ideal of the relations on $U_i$ that are relatively simple. Our approach involved looking for equations with relatively few terms and picking ones of low length, until they generate the space of equations.
\end{rmk}

\begin{rmk}\label{C3rot}
The residual $C_2$ action has variables $U_0,\ldots, U_3$ as even and $U_4,\ldots, U_7$ as odd. There is also a $C_3$ action
$$
(U_0:\ldots:U_7) \mapsto (U_0:U_2:U_3:U_1:U_5:U_6:U_4:U_7)
$$
The coordinates $U_0,\ldots, U_3$ can be viewed as sections of the bicanonical class on $W_{\Gr}/C_{14}$. There is a dimension three space of even degree two relations, which are not predicted by the Hilbert function considerations.  There are additional spaces of degree three relations of dimension four (even) and two (odd).  This gives us a total of nine equations in $8$ homogeneous variables. 
The size of the equation file is about $30$Kb, so, while it is small by computer standards, it is not worth presenting in printed form.
\end{rmk}

Importantly, equations of $X$ in terms of $U_j$  have coefficients in the field of $\QQ(\sqrt{-15},\sqrt{-7})$. However, we observed that it was possible to make a linear change of variables to new coordinates $(W_0,\ldots,W_7)$ so that the equations are defined over  $\QQ(\sqrt{-15})$, with the file size of about $20$Kb \cite[EquationsOfXintermsofW]{support}. \footnote{We list these equations in the Appendix, but they are not exactly human-usable.}
The linear change was designed to send the six singular points (whose coordinates are complicated in the $U$-basis) to 
$$(0:1:0:0:\pm 1:0:0:0),~(0:0:1:0:0:\pm 1:0:0),~(0:0:0:1:0:0:\pm 1:0)$$
and to make the tangent spaces at the points defined over $\QQ(\sqrt{-15})$. 
The $C_3$ action on $W$ is again given by 
$$
\sigma:(W_0\,\ldots, W_7) \to (W_0,W_2,W_3,W_1,W_5,W_6,W_4,W_7).
$$
Unfortunately, the equations are still too lengthy to be presented in a printed version of the paper.

\begin{rmk}
We have verified that these equations cut out the scheme with the expected Hilbert polynomial $h(n) = 12n^2 - 6n + 2$ of a surface with $K^2=6$ in its bicanonical  embedding for $n\geq 5$, see \cite[EquationsOfXmagma]{support}.  
\end{rmk}

\begin{rmk}
The embedding by $W_j$ is not projectively normal. In particular, not all degree two sections of $\mathcal O(2)$ can be written as quadratic polynomials in $W_i$. This makes working with sections of $\mathcal O(2)$ more complicated, as one needs to represent some 
of them as rational functions  in $W_i$ with degree three numerators and degree one denominators.
\end{rmk}

\section{Finding the triple cover: the key step}\label{sec.4}
This section is devoted to the key step of the construction of the fake projective plane $\PP^2_{fake}$ which is a Galois triple cover of $W_{\Gr}/C_{14} = X/C_2$. 
Finding a smooth Galois triple cover of a surface with singular points of type $A_2$ involves finding Weil divisors which are not Cartier which could in general be difficult. This was an extremely delicate calculation that took about six months and multiple dead ends before introduction of several important ideas. The process is described below and is implemented in the accompanying Mathematica file \cite[Section4.nb]{support}.

\medskip
{\bf The approach.} 
Suppose that we have a fake projective plane $\PP^2_{fake}$  with a $C_3\times C_3$ automorphism group such that  the quotient by the first $C_3$ 
is isomorphic to the quotient of $X$ 
$$
\PP^2_{fake} / C_3 \cong X/C_2
$$
with the above $C_2$ action. The second $C_3$ induces the action on $X/C_2$ coming from the above permutation $\sigma$ of $W_i$ coordinates.
Let $4H$ be a divisor on $\PP^2_{fake}$ where $H$ is satisfies $3H=K_{\PP^2_{fake}}$. We may moreover assume that the automorphism group $C_3\times C_3$ of $\PP^2_{fake}$ fixes $4H$ and therefore produces a projective action on $H^0(\PP^2_{fake},4H)\cong \C^3$.
It can be shown that the action of the second $C_3$ permutes the eigenvectors $u_0,u_1,u_2$ of the first $C_3$, with weights
 $1, \exp({\frac {2 \pi {\rm i}}3}), \exp({-\frac {2 \pi {\rm i}}3})$.

 \medskip
The cubes $u_0^3, u_1^3, u_2^3$ will be sections of $12H = 4K$ which will descend to 
 $C_2$ invariant (=even) sections of $\mathcal O(2)$ on $X$.
 In view of the $C_3$ action, these will be $f,\sigma(f),\sigma^2(f)$ for some section $f$. Importantly, the product $u_0 u_1 u_2$ will be an even $C_3$ invariant section $d$ of $\mathcal O(2)$ and there must hold
 \begin{equation}\label{eq.key}
f  \sigma(f)  \sigma^2(f) = d^3.
\end{equation}
Knowledge of $f$ and $d$ allowed us to find the triple cover fairly easily as is explained in the next section.

\medskip
{\bf How we found $f$ and $d$.} In addition to the equation \eqref{eq.key}, we know that $\{f=0\}$ must pass through $(0 : 1 : 0 : 0 : 1 : 0 : 0 : 0) $ and  $(0 : 0 : 1 : 0 : 0 : 1 : 0 : 0)$. Its proper preimage on the blowup of the singularity will have multiplicity $2$ and $1$, and $1$ and $2$ on the corresponding pairs of exceptional curves. In the notations of the diagram below, the preimage of the divisor $\{f=0\}$ on $X$ on the minimal resolution is given by $3C + 2A_1 + B_1 + 2B_2 +A_2$, where 
$A_i, B_i$ are the exceptional curves. The preimage of the divisor $\{d=0\}$ is $C+\sigma(C)+\sigma^2(C) + \sum_{i=1}^3(A_i+B_i)$.

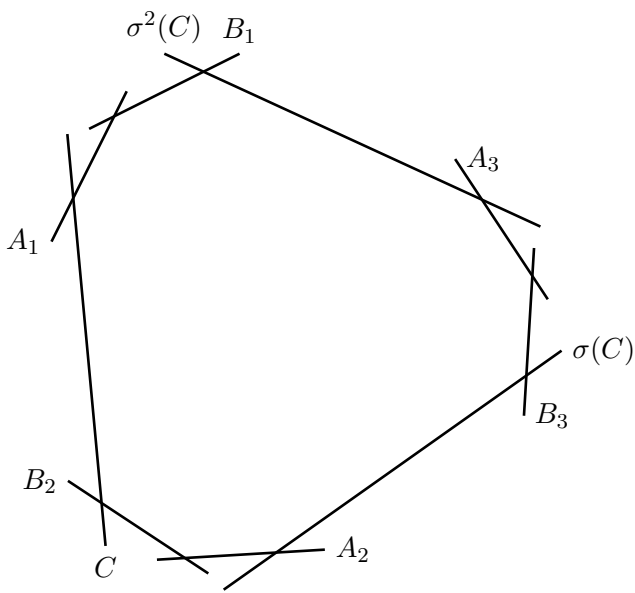
\medskip
\begin{tikzpicture}

    \draw[line width=1pt] (-2,3)--(-3,1) node[left]{$A_1$};
    \draw[line width=1pt] (-2.5,2.5)--(-.5,3.5) node[above]{$B_1$};
 
 \draw[line width = 1pt] (3.5,1.2)--(-1.5,3.5)node[above]{$\sigma^2(C)$};
    
    \draw[rotate=-120][line width=1pt] (-2,3)--(-3,1) node[right]{$A_3$};
    \draw[rotate=-120][line width=1pt] (-2.5,2.5)--(-.5,3.5) node[right]{$B_3$};
    
 \draw[rotate=-120][line width = 1pt] (3.5,1.2)--(-1.5,3.5)node[right]{$\sigma(C)$};
     
    \draw[rotate=120][line width=1pt] (-2,3)--(-3,1) node[right]{$A_2$};
    \draw[rotate=120][line width=1pt] (-2.5,2.5)--(-.5,3.5) node[left]{$B_2$};
 
 \draw[rotate=120][line width = 1pt] (3.5,1.2)--(-1.5,3.5)node[below]{$C$};
   
\end{tikzpicture}

\begin{rmk}
The automorphism group of the minimal resolution of $X$ does not switch the exceptional lines over a singular point, so there is a meaningful choice here which of the lines is $A_1$ and which is $B_1$, with the other choices then fixed by $C_3$ symmetry. 
\end{rmk}

\begin{rmk}\label{notCartier}
The divisor of $f=0$ on $X/C_2$ is $3C$ where $C$ is a Weil divisor which is not Cartier. Indeed, its preimage on the blowup of $X/C_2$ is numerically equivalent to $C + \frac 23 A_1 + \frac 13B_1 + \frac 23B_2 + \frac 13A_2$ 
in the notations of the above diagram.
\end{rmk}

To solve \eqref{eq.key},
we start by computing  the dimension $19$ space of sections $H^0(\PP^2_{fake} / C_3, 4K_{\PP^2_{fake} / C_3})$.
These are realized as elements of  $H^0(X,\mathcal O_X(2))$ which are even with respect to the covering involution.
A dimension $17$ subspace of $H^0(X,\mathcal O_X(2))$ is given by the even quadratic polynomials in $(W_0,\ldots, W_7)$ subject to three quadratic relations of \cite[EquationsOfXintermsofW]{support}. We augment it to the whole space by calculating two additional basis elements of the form $P(W)/(W_4+W_5+W_6)$ where $P$ is a degree three polynomial in $W_i$ which vanishes on $\{0=W_4+W_5+W_6\}\cap X$ and is odd with respect to the involution.

\medskip
The requirement of $\{f=0\}$ passing  through the singular points as above reduces the dimension of the space of sections $f$ from $19$ to $13$. Similarly, $\{d=0\}$ must pass through all three singular points, which reduces the dimension  from $\dim H^0(X/C_2,\mathcal O(2))^{C_3}= 7$ to $6$. So the pair $(f,d)$ can be described by $13+6=19$ parameters, up to scaling.

\medskip
The equation \eqref{eq.key} is cubic in the coordinates of $f$ and $d$. It takes place in the space $H^0(X/C_2,\mathcal O(6))^{C_3}$ of dimension $67$, so we have $67$ cubic equations  in $19$ variables. with coefficients  in $\QQ[\sqrt{-15}]$. The above linear conditions on $f$ and $d$ reduced these $67$ equations to $58$. We then used Smith decomposition and a version of LLL algorithm (both built into Mathematica) to find a basis of the space of equations with smaller coefficients. Specifically, given a list of polynomials with 
coefficients in $\Z[\sqrt{-15}]$ we find find the corresponding matrix of real and imaginary parts of these equations and the ones multiplied by $\sqrt{-15}$, then find a small basis of the saturation of its row span and transform it back to the equations. The resulting  file of $58$ cubic equations \cite[58CubicRels]{support} was approximately $3.3$Mb long. It was too big to be solved by Mathematica or Magma, however there were further simplifications that allowed us to do it.

\medskip
If $p$ is a fixed point of the $C_3$ action on $X/C_2$, then we see that $f^3(p)=d^3(p)$. This cubic equation can be reduced to a linear equation (in one of three ways).  There are three fixed points, and we picked the conditions which had coefficients in $\QQ[\sqrt{-15}]$. Only one of several possible choices led to an eventual solution. This reduced the number of unknowns to $19-3=16$ and the number of cubic equations to $55$. 

\medskip
We computed the neighborhoods of the blowups of exceptional lines. At a preimage of a singular point with two exceptional lines $A_1$ and $B_1$ the section 
$f=u_0^3$ must locally look like $B_1+ 2 A_1 + 3C$ where $C$ is some curve that intersects $A_1$ at a point. 
The section $d$ locally looks like $A_1 + B_1 +C$. Starting from the equations of $f$ and $d$ we can compute the restrictions of $3C$ and $C$ to the exceptional curve $A_1$, respectively.  The  polynomial coming from $f$ is up to a constant the cube of the polynomial coming from $d$, which leads to $6$ quadratic equations for each of the two points  $(0 : 1 : 0 : 0 : 1 : 0 : 0 : 0) $ and  $(0 : 0 : 1 : 0 : 0 : 1 : 0 : 0)$.  Specifically, if the restrictions of $f$ and $d$ to the exceptional curve are $a_0 + a_1 t + a_2 t^2 + a_3 t^3$ and $b_0 + b_1 t$ respectively, then the equations  are 
$$0=3 a_3 b_0 - a_2 b_1= a_2 b_0 - a_1 b_1=a_1 b_0 - 3 a_0 b_1=
   9 a_0 a_3 - a_1 a_2= 3 a_0 a_2 - a_1^2= 3 a_1 a_3 - a_2^2.$$
We have thus constructed a system of $12$ quadratic and $55$ cubic equations in $16$ variables, with coefficients in $\QQ[\sqrt{-15}]$.
\emph{The size of the file \cite[Rels23withrr] {support} containing these equations was approximately $5.1$Mb, and standard Mathematica software and basic hardware were not fast enough to solve them.}

\medskip
We then employed a natural, yet amusing, trick. Magma has readily computed the Hilbert polynomial of the reduction of the above system modulo $19$ with $\sqrt {-15}$ set as $2 \hskip -5pt \mod 19$. It was precisely $1$, which meant that the system had a unique solution. By adding linear relations and rechecking the Hilbert polynomial, it was easy to find that solution modulo $19$. This then allowed us to inductively find a solution modulo $19^k$, since increasing $k$ by $1$ leads to simple linear equations modulo $19$ (we used an appropriate square root of $(-15)$). After computing it up to $k=200$, we had enough information to find $f$ and $d$, on the assumption that the coefficients had reasonable numerators and denominators. We verified that they satisfy the $12$ quadratic and $55$ cubic equations precisely by a symbolic calculation.
Specifically, we obtained that $d$, up to scaling, is
$$
{
\tiny
\begin{array}{l}
42318123032 W_0^2 + 
 2256004 (-23709 - 33355 \ii\sqrt{15}) W_0 (W_1 + W_2 + W_3) 
    + 
 4512008 (-369999 - 115101 \ii\sqrt{15}) (W_1 W_2 
 + W_1 W_3 + W_2 W_3) 
    \\[.2em]
+    4512008 (-134064 - 81144 \ii\sqrt{15}) (W_1^2 + W_2^2 + W_3^2 - W_4^2 - 
    W_5^2 - W_6^2) + 
 6008 (-60345558 - 90294750 \ii\sqrt{15}) (W_4 + W_5 + 
    W_6) W_7 
     \\[.2em]
     + (457763819877 - 572077298835 \ii\sqrt{15}) W_7^2
 \end{array}
   }
$$
while $f$ is given by a notably more complicated formula.

%%%%%%%%%%%%%%%%%%%%%%%%%%%%%%%%%%%%%%
\section{Finding the triple cover and verification of FPP claim}\label{sec.5}
In this section we discuss the construction of the triple cover and the verification that it is indeed a fake projective plane. We recall that we have constructed a surface $X$ in $\PP^7$ cut out by three quadratic and six cubic equations in coordinates $(W_0:\ldots:W_7)$ in the Appendix. It has a free action of $C_2$
$(W_0:\ldots:W_7)\to (W_0:\ldots:W_3:-W_4,\ldots:-W_7)$. It is also acted on by a cyclic group of order three  given by
$$
\sigma(W_0:\ldots:W_7) = (W_0:W_2:W_3:W_1:W_5:W_6:W_4:W_7).
$$
The quotient $X/C_2$  has three singular points of type $A_2$, permuted by $C_3$. We have also found an even section $f$ of $\mathcal O_X(2)$ and an even $C_3$-invariant section of $\mathcal O_X(2)$ such that $f \sigma(f)\sigma^2(f) = d^3$ and
the divisor of $f$ is $3C$ where $C$ is Weil but not Cartier, see Remark \ref{notCartier}.

\medskip
Once we have found $f$ and $d$, constructing equations of the triple cover $\PP^2_{fake}$ is fairly straightforward, similar to \cite{BorKeum}. We consider the normalization $Z$ of $X/C_2$ in the field obtained by attaching an algebraic function $z$ which satisfies
\begin{equation}\label{z}
z^3 = \sigma (f) f^{-1}.
\end{equation}
This field has a $C_3 \times C_3$ automorphism group. Namely, there is an action of the \emph{covering} $C_3$ given by 
$$
z\mapsto {\rm e}^{\frac 23 \pi \ii} z, ~ g\mapsto g~{\rm for~}g\in {\rm Rat}(X/C_2).
$$
There is also a commuting action of the \emph{lift} $C_3$ given 
by 
$$
z\mapsto  d\, \sigma(f)^{-1} z, ~ g\mapsto \sigma(g)~{\rm for~}g\in {\rm Rat}(X/C_2).
$$
Indeed, we get 
$$z^3\mapsto d^3 \,\sigma(f)^{-3} z^3 =( f \sigma(f) \sigma^2(f))  \sigma(f)^{-3} \sigma(f) f^{-1} = \sigma^2(f)\sigma(f)^{-1} = \sigma(\sigma(f) f^{-1})
$$
so \eqref{z} is preserved by the above action.

\begin{rmk}\label{zdiv}
The usual convention is that the action of $C_3$ on points of $X$ is induced by
$$\sigma(g) (p) = g(\sigma(p))$$
so
$$
\sigma(f)(p) = 0 \iff f( \sigma(p)) = 0 \iff \sigma(p)\in C \iff p\in \sigma^2(C).
$$
Therefore, the divisor of $\sigma(f)$ is $3\sigma^2(C)$ and the divisor of $\sigma^2(f)$ is $3\sigma(C)$. The (Weil) divisor of $z$ on $X/C_2$ makes sense and is given by $\sigma^2(C) - C$.  
\end{rmk}

\medskip
Since we know what $\{f=0\}$ looks like locally at singular points, we see that the cyclic triple cover $Z$ given above is smooth and $X/C_2$ is $Z/C_3$ for the covering $C_3$. We are interested in describing $Z$ explicitly by the equations on sections of $H^0(Z,2K_Z)$. 

\medskip
Note that $2K_{Z/C_3}$ is $\mathcal O(1)$. The covering $C_3$ induces an action on $H^0(Z,2K_Z)$ which splits it into three eigenspaces. The invariant subspace is naturally identified with the global sections of $H^0(X/C_2, \mathcal O(1))$ and has a basis $\{W_0,W_1,W_2,W_3\}$.

\medskip
The additional sections in $H^0(Z,2K)$  for the two other eigenspaces of the covering $C_3$ action on $Z$ can be thought of as  spaces of  $z g$ or $z^{-1} g$ with $g$ a meromorphic section of $\mathcal O(1)$ on $Z$ which is a pullback of one from $X/C_2$. For the first of these eigenspaces, the condition of holomorphicity of $zg$ is 
 $\ord_D(g) \geq -\ord_D(z)$ for all divisors $D$. This means that $g$ is a section of the reflexive sheaf 
 $\mathcal O(1)( \sigma^2( C ) - C)$ 
on $X/C_2$, in view of Remark \ref{zdiv}. As a consequence, $g \,d$ will be a section of $\mathcal O(3)( \sigma^2(C) - C)$ which vanishes at 
$C + \sigma(C) + \sigma^2(C)$, i.e. a section of $\mathcal O(3)(- 2C - \sigma(C))$. This  can be identified as the linear subspace of $H^0(X,\mathcal O(3))^{C_2}$ cut out by conditions of vanishing twice at $C$ and once at $\sigma(C)$.
These are readily calculated numerically by finding a number of random points on these curves.
The other eigenspace of the covering $C_3$ is determined similarly, and we see that 
$H^0(Z,2K_Z)$ is naturally identified with
$$
H^0(X,\mathcal O(1))^{C_2} \bigoplus 
\Big(z\, d^{-1}H^0(X,\mathcal O(3)(-2C - \sigma(C)))^{C_2} \Big)
 \bigoplus 
\Big(z^{-1} d^{-1}H^0(X,\mathcal O(3)(-2\sigma^2(C) - \sigma(C)))^{C_2} \Big)
$$
One can write these sections as rational functions in $W_0,\ldots, W_7$ of total degree one and evaluate them on
points of $X$. When looking for relations among them, one must ensure that the total degree in $z$ is divisible by $3$ and that $z^3$ is converted into $\sigma(f) f^{-1}$ (or alternatively one can construct multiple points on $Z$).

\medskip
We would like to have the basis of $H^0(Z,2K_Z)$ which is nice with respect to the lift $C_3$ action. Observe 
that 
$$
\begin{array}{rl}
\sigma_{\rm lift}\Big(
z \,d^{-1}H^0(X,\mathcal O(3)(-2C - \sigma(C))^{C_2}
\Big)
&=
(z  \,d\,  \sigma(f)^{-1} )d^{-1}H^0(X,\mathcal O(3)(\sigma^2(-2C - \sigma(C)))^{C_2}
\\
&=z\,d^{-1} \left( d \,\sigma(f)^{-1}\right) H^0(X,\mathcal O(3)(-2\sigma^2(C) - C))^{C_2}
\\
&=z \,d^{-1} H^0(X,\mathcal O(3)((2\sigma^2(C) - \sigma(C)- C )-2\sigma^2(C) - C))^{C_2}
\\
&=z \,d^{-1} H^0(X,\mathcal O(3)(- 2 C - \sigma(C))^{C_2}
\end{array}
$$
and similarly for $z^{-1} d^{-1}H^0(X,\mathcal O(3)(-2\sigma^2(C) - \sigma(C)))^{C_2} $.
So we can take one of the elements of the eigenspace and make the others by applying $\sigma_{\rm lift}$. 
This gives us a basis
$(P_0,\ldots, P_9)$ of $H^0(Z,2K_Z)$ which has $C_3\times C_3$ action given by
\begin{align*}
\sigma_{\rm covering} (P_0,\ldots,P_9) =& \,(P_0,P_1,P_2,P_3,  {\rm e}^{\frac 43 \pi \ii}P_4,  {\rm e}^{\frac 43 \pi \ii} P_5,  {\rm e}^{\frac 43 \pi \ii} P_6,  {\rm e}^{\frac 23 \pi \ii} P_7,  {\rm e}^{\frac 23 \pi \ii} P_8,  {\rm e}^{\frac 23 \pi \ii}P_9),
\\
\sigma_{\rm lift}(P_0,\ldots,P_9) =& \,(P_0,P_2,P_3,P_1,P_5,P_6,P_4,P_8,P_9,P_7).
\end{align*}

\medskip
We have implemented the above in \cite[Section5.nb]{support}. 

\begin{rmk}
We were hampered slightly by the lack of projective normality of $X$. In fact, only a codimension two subspace of even sections of $H^0(X,\mathcal O(3)$ can be written as a polynomial in $W_0,\ldots, W_7$. Fortunately, this was still enough to find one element in each of the three-dimensional eigenspaces of $H^0(Z,2K_Z)$, and then $\sigma_{\rm lift}$ gave us the basis of the space.
\end{rmk}

 \begin{rmk}
 As has been observed before, one of the difficulties is finding a good basis of sections of $2K$ and a good basis in the space of equations, in order to have the coefficients of manageable length. After we found the equations, we have made a linear change of variables with coefficients in $\QQ[\sqrt{-15}]$ to get a nicer description of the fixed points of $C_3\times C_3$ while preserving the shape of the group action. This has only lead to a moderate improvement in the size of the coefficients. This is also implemented in \cite[Section5.nb]{support}, with the resulting file  \cite[EqsFppd3D3]{support}
 \end{rmk}

\begin{rmk}
We also use the knowledge of $f$ and $d$ to construct the double cover of 
 $\PP^2_{fake}$  obtained from $X$ by  attaching the above function $z$. This double cover is given in its $2K$ embedding by 
 $20$ variables and $100$ quadratic equations. It will be very useful in the next section.
\end{rmk}

Since our computations often involved approximate points, and one should generally be wary of long computer code (that took a fair bit of time to debug), we have spent some time directly verifying that our $84$ cubic equations cut out a fake projective plane. 
Specifically, we saved the equations in the Magma format and over the field 
$\QQ(\sqrt {-15})$ that the $84$ cubic relations  cut out a scheme with Hilbert polynomial $p(n) = 18 n^2 - 9 n +1$, as expected. We then verified that they generate a prime ideal $I$ and that $H^1(Z,\mathcal O_Z)=H^2(Z,\mathcal O_Z)=0$
by working over a finite field and using semicontinuity.

\medskip
We also verify smoothness as follows. The scheme in question is smooth if and only if the radical of the ideal generated by $I$ and the $7\times 7$ minors of the $84\times 10$ Jacobian matrix $Jac$ of the generators of $I$ is 
the irrelevant ideal. It is is impossible to calculate all of the minors in any reasonable amount of time, even over a finite field. We used the following trick. We picked random $84\times 84$ and $10\times 10$ matrices $A$ and $B$ over a finite field and looked at the first $7\times 7$ minor of the matrix $ A \,Jac \, B$ to get a reasonably generic linear combination of the minors of $Jac$.  We repeated it three times, added the resulting minors to $I$ and verified that the resulting ideal has zero Hilbert polynomial. We had to use a powerful computer cluster (the "Galois" server based at Warwick University) in order to do this calculation in a reasonable amount of time. The relevant Magma code is in \cite[EqsFPPd3D3Magma  and  EqsFPPd3D3MagmaFinite]{support}.

\medskip
Once we know that $Z$ is a smooth surface with $H^1(Z,\mathcal O_Z)=H^2(Z,\mathcal O_Z)=0$, the Hilbert polynomial implies that the hyperplane class $D$ on $Z$ satisfies $D^2 = 36, K_Z D =  18$. To show that $Z$ is a fake projective plane and $D$ is twice the canonical class we only needed to show that $\chi(Z,\mathcal O(2K_Z))=10$
and $h^0(Z,\mathcal O_Z(D-K))=0$  as in \cite{BorKeum}.  We used the "Galois" server and Macaulay2 for this calculation.

%\section{Computations with the equations.}
%In this section we explain how we used the equations to check whether there are any curves of degree less than canonical on $\PP^2_{fake}$
%{\bf This is still not done, but it is just a matter of computation. It is important to do, because in the process we will write down homogeneous ideals of curves that are sections of $\mathcal O(4H)$ and curves that are sections of $K_{FPP} + {\rm torsion}$}. 

\section{Constructing five more pairs of fake projective planes}\label{sec.6}
Cartwright and Steger \cite{CSlist} have discovered the following ball quotients in the same commensurability class as the fake projective plane we have found so far. The labels are the shorthand for their more extensive notation. For example, the FPP we have found is 
FPP: (C2,p=2,$\emptyset,d_3 D_3)$ in the notation of \cite{CSlist} and is $d_3D_3$ in the diagram below.
The quotient of the complete intersection of Grassmannian in Section \ref{sec.3} is $D_3$.

\smallskip
All of the arrows correspond to degree three maps. The thicker arrows indicate Galois covers, i.e. quotient maps for some $C_3$ action.

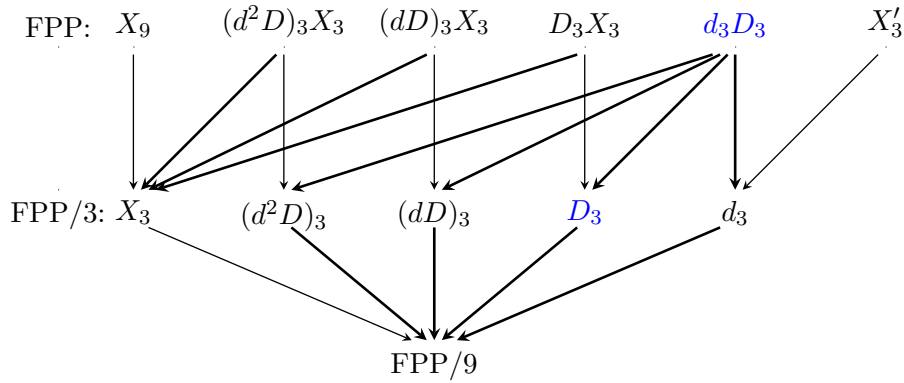
\begin{tikzpicture}

    \draw (-1,4)--(-1,4) node[above] {FPP:};
    \draw (-1,2.1)--(-1,2.1) node[below] {FPP/3:};
    
   \draw (0,4)--(0,4) node[above] {$X_9$};
   \draw (2,4)--(2,4) node[above] {$(d^2D)_3X_3$};
   \draw (4,4)--(4,4) node[above] {$(dD)_3X_3$};
   \draw (6,4)--(6,4) node[above] {$D_3X_3$};
   \draw (8,4)--(8,4) node[above] {\textcolor{blue}{$d_3D_3$}};
   \draw (10,4)--(10,4) node[above] {$X_3'$};   
     
  \draw[line width=.5pt,-stealth] (0,3.9)--(0,2.1) node[below]{$X_3$};
  \draw[line width=.5pt,-stealth] (2,3.9)--(2,2.1) node[below]{$(d^2D)_3$};
  \draw[line width=.5pt,-stealth] (4,3.9)--(4,2.1) node[below]{$(dD)_3$};
  \draw[line width=.5pt,-stealth] (6,3.9)--(6,2.1) node[below]{\textcolor{blue}{$D_3$}};
  \draw[line width=1pt,-stealth] (8,3.9)--(8,2.1) node[below]{$d_3$};
    \draw[line width=.5pt,-stealth] (9.9,3.9)--(8.1,2.1);
  
   \draw[line width=1pt, -stealth] (1.9,3.9)--(0.1,2.1) ;
    \draw[line width=1pt,  -stealth] (3.9,3.9)--(0.2,2.1) ;
     \draw[line width=1pt,  -stealth] (5.9,3.9)--(0.3,2.1) ;
     
    \draw[line width=1pt, -stealth] (7.7,3.9)--(2.1,2.1) ;
    \draw[line width=1pt,  -stealth] (7.8,3.9)--(4.1,2.1) ;
     \draw[line width=1pt,  -stealth] (7.9,3.9)--(6.1,2.1) ;
        
    \draw[line width=.5pt,-stealth] (0.2,1.6)--(3.7,.1) ;
    \draw[line width=1pt,  -stealth] (2.1,1.6)--(3.9,.1) ;
    \draw[line width=1pt,  -stealth] (4.0,1.6)--(4.0,.1) node[below] {FPP/9};
    \draw[line width=1pt,  -stealth] (5.9,1.6)--(4.1,.1) ;
    \draw[line width=1pt,  -stealth] (7.8,1.6)--(4.3,.1) ;
   
\end{tikzpicture}

We will now describe the main idea used to construct the other surfaces in the above diagram and, in particular, all of the other fake projective planes in this class.

\medskip
Suppose that we have constructed (i.e. have explicit equations of) a fake projective plane $P_1$ with a $C_3$ action and suppose that there is another FPP $P_2$ which covers the quotient of $P_1/C_3$ as a \emph{non-Galois} triple cover. Then we can try to construct $P_2$ as follows.

\medskip
At the level of the fundamental groups, we have 
$P_i= \mathcal B^2/\Gamma_i$ where $\Gamma_1$ and $\Gamma_2$ are index three subgroups of a larger group $\Gamma$ that corresponds to $P_1/C_3$. However, $\Gamma_1$ is a normal subgroup of $\Gamma$ and $\Gamma_2$ is not. The kernel of the action of $\Gamma$ on the cosets of $\Gamma_2$ is a normal subgroup $\Gamma_3$ of $\Gamma$ of index six with $\Gamma_3/\Gamma \cong S_3$, which is contained in $\Gamma_2$. It is not contained in $\Gamma_1$  and, therefore, the intersection $\Gamma_4=\Gamma_1\cap \Gamma_3$ is a normal subgroup of $\Gamma$ of index $18$ which is contained in both $\Gamma_1$ and $\Gamma_2$. 

\medskip
In terms of the surfaces, we have a smooth surface $P_4=\mathcal B^2/\Gamma_4$ which is a six-fold unramified cover of both $P_1$ and $P_2$. There is an action of $S_3\times C_3$ on $P_4$ such that $P_1$ is the quotient of $P_4$ by $S_3$ and $P_2$ is the quotient of $P_4$ by $C_2\times C_3$ where $C_2$ is a subgroup of $S_3$. Both of these quotients are unramified. We will also consider a double cover
of $P_1/C_3$ which corresponds to taking quotient of $P_4$ by $C_3\times C_3$. We will denote this (singular) surface by $\widehat {P_1/C_3}$.

\medskip
We have, $\chi(P_4,\mathcal K_{P_4}) = 6 \chi(P_1,\mathcal K_{P_1}) = 6$. 
We will assume \footnote{This is a natural assumption, which gets justified a posteriori by the success of the process.}
 that the surface $P_4$ is regular, so the dimension of $V=H^0(P_4,K_{P_4})  $ is exactly $5$.
The Holomorphic Lefschetz formula implies that the traces of the action on $V$ of the nonidentity elements of $S_3$ and $C_3$ are $-1$ due to the trivial action on $H^2(P_4, K_{P_4})$. Therefore, as an $S_3$ representation, $V$ is isomorphic to the direct sum
$$
V= V_1 \oplus V_2 \oplus V_2'
$$
where $V_1$ is  the one-dimensional sign representation and $V_2$ and $V_2'$ are two copies of the dimensional representation of $S_3$. Since the actions of $C_3$ and $S_3$ commute, the $C_3$-eigenspaces must be representations of $S_3$. The dimensions of weight $1, w, w^2$ eigenspaces are $1,2,2$ respectively, so we may assume that $V_1$ is trivial under $C_3$ action and $V_2$ and $V_2'$ are eigenspaces of weight $w$ and $w'$. 

\medskip
The dimension two representation of $S_3$ has a basis $(r_1,r_2)$ such that $(1,2,3) r_1 = w r_1$, $(1,2,3) r_2 = w^2 r_2$, $(1,2) r_1 = r_2$, $(1,2) r_2 = r_1$.  Consider such basis $(r_1, r_2)$ of $V_2$ and similarly $(r_1',r_2')$ of $V_2'$. The following observation is key.

\begin{proposition}\label{si}
In these notations,  $s_1=r_1 r_2'$ and $s_2= r_2 r_1'$
are (pullbacks of) elements of $H^0(\widehat {P_1/C_3},2K_{\widehat {P_1/C_3})})$.
Moreover, the covering involution on $\widehat {P_1/C_3}\to  {P_1/C_3}$ permutes them.
Similarly,
$s_3=r_1 r_2$ and $s_4=r_1' r_2'$ are (pullbacks of) sections of $H^0(P_1, 2K_{P_1})$
of weights $w^2$ and $w$ with respect to the $C_3$ action on $P_1$.
The following equation holds on
$P_1/C_3$:
$$
s_1s_2 =  s_3 s_4.
$$
\end{proposition}

\begin{proof}
This statement follows immediately from the description of the group action on $r_i$.
\end{proof}

\medskip
Proposition \ref{si} gives us a way to find the divisors of  $r_i$ and $r_i'$.
Namely, we can solve  for all linear relations on 
$$
\Sym^2 \Big(H^0(\widehat {P_1/C_3},2K_{\widehat {P_1/C_3}})\Big) \bigoplus \Big(
H^0(P_1, 2K_{P_1})_{w^2}\otimes H^0(P_1, 2K_{P_1})_{w}
\Big)
$$
and impose the conditions that the right hand side is decomposable (the corresponding matrix is of rank one) and the left hand side is of rank at most two.
This allows us to find $s_1,\ldots, s_4$ and then the divisors of $r_1,\ldots, r_2'$. For example, we can find $r_1=0$ as the intersection of $s_1=0$ and $s_3=0$.

\medskip
Then we can construct the surface $P_4$ and finally get $P_2$ as its quotient. Specifically, the difference between the divisors of $r_1$ and $r_2$ is an order three torsion bundle on the double cover of FPP/3. In order to construct the FPP, we will consider the sections of $2K$ on this triple cover as follows.
We can look at sections of $3K$ that vanish on $r_1=0$ and others at $r_2=0$.  This requires constructing sections of $3K$. One way is to look at sections of $4K$ which are zero on the anti-invariant section of $K$ (all of this is done on the double cover  $\widehat {P_1/C_3}$ of FPP/3).

\medskip
We have used this method successfully to construct five more pairs of fake projective planes, see \cite[Section6D3X3.nb,
Section6twin.nb,
Section6X3prime.nb, Section6X9.nb
]{support}.
Specifically, we used the $C_3\times C_3$ action on the FPP $d_3D_3$ to get $D_3X_3$ (and its unramified double cover), as well as $(dD)_3X_3$ and $(d^2D)_3X_3$ (twin FPPs) and $X_3'$. Then we used $D_3X_3$ surface to get $X_9$. 
The verification of the smoothness and the fact that these are indeed FPPs was done similarly to the $d_3D_3$  case in Section \ref{sec.5}.

\begin{rmk}
There were some technical issues that we were not able to resolve to our complete satisfaction. 
Specifically, our method a priori produces equations over an algebraic extension of the original field $\QQ(\sqrt{-15})$, often by adding the cube root of unity $w$. We were able to find an appropriate linear change of coordinates to get the field to be $\QQ(\sqrt{-15})$ for  $D_3X_3$. The twin pairs of 
$(d^2D)_3X_3$ and $(dD)_3X_3$ are defined over $\QQ(\sqrt{-15},\sqrt{-3})$, and we are unable to distinguish one from the other with our method. We are hopeful that the FPPs $X_3'$ and $X_9$ can be defined over $\QQ(\sqrt{-15})$, but we were unable to find a linear change of the coordinates to do so. We were also unable to successfully control the size of the coefficients. For example, coefficients for $X_9$ are several thousand decimal digits long in the natural basis $\{1,\sqrt 5,\sqrt{-3},\sqrt{-15}\}$.
\end{rmk}

\section{Further directions}\label{sec.7}
In this section we list several open problems that we have not addressed in our research, together with plausible approaches to them.

\medskip
\begin{itemize}
\item
It would be interesting to find a way to distinguish between the two surfaces $(d^2D)_3 X_3$ and $(dD)_3 X_3$ and to match our equations with the calculations of Cartwright and Steger. One approach could be based on finding the fundamental group of our FPPs. It can be done by picking a base point, then making a hyperplane cut and looking at generators coming from the fundamental groups of the corresponding complex curves. However, we have not attempted it and do not know if this is feasible.
\\[0.2em]

\item 
Unfortunately, we were only able to construct the equations of the fake projective planes $X_3'$ and $X_9$ with coefficients in the field
$\Q(\sqrt{-15},\sqrt{-3})$. However, it may be possible to find such equations with coefficients in
$\Q(\sqrt{-15})$. In a similar vein, we would like to be able to find shorter equations of all of the surfaces involved, even though this might be impossible.\\[.2em]

\item 
It has been conjectured that fake projective planes do not have any effective curves in classes $H$ and $2H$, up to torsion. This question can, in theory, be addressed for the fake projective planes we have constructed in this paper as follows. As a byproduct of our computation of the triple cover $d_3D_3\to D_3$, we have found the sections of $4H$ on the fake projective plane $d_3D_3$. We can also trace the $C_{14}$ torsion in the Picard group of this fake projective plane to the Grassmannian construction.  This should, in theory, allow one to verify the conjecture for $d_3D_3$. We have not attempted to verify the torsion calculations of Cartwright and Steger for the other FPPs.
One approach is the following. If a fake projective plane has a double cover $Y\to \PP^2_{fake}$, then all but one nontrivial torsion divisor classes $L$ on $\PP^2_{fake}$ give rise to the spaces $H^0(Y, K_Y+L)$  and $H^0(Y, K_Y-L)$ of dimension at least two, with each eigenspace of dimension at least one (the dimensions of the eigenspaces are exactly one iff $h^1(Y,K_Y\pm L )=0$). If we pick sections 
$u_{\pm L,\pm}$ in the corresponding eigenspaces, then 
$$
s_1 := u_{L,+} u_{-L,+},~s_2 := u_{L,-} u_{-L,-},~s_3 := u_{L,+} u_{-L,-},~s_4 := u_{L,-} u_{-L,+}
$$ 
satisfy 
\begin{equation}\label{eq1234}
s_1 s_2 = s_3 s_4
\end{equation}
with $s_1$ and $s_2$ even sections of $H^0(Y, K_Y)$ and $s_3$ and $s_4$ odd sections of  $H^0(Y, K_Y)$. Thus we may look for such relations to uncover torsion line bundles. It is not clear if this approach is feasible in practice. For example, one can define  $t_1=s_1+s_2, t_2=s_1-s_2, t_3=s_3+s_4, t_4=s_3-s_4$ to rewrite \eqref{eq1234} as 
$$t_1^2 - t_2^2 = t_3^2 - t_4^2.$$
Each $t_i$ is in a linear space of dimension $10$, so we end up with quadratic equations in $40$ variables. If $\PP^2_{fake}$ has additional symmetries one could likely  restrict their attention to some eigenspaces in \eqref{eq1234}.\\[.2em]

\item
This paper got its start in the construction of surfaces with $p_g=q=0,~K^3=3$ and fundamental group $C_{14}$. There are many other surfaces with these numerical invariants but a different fundamental group that also come from smoothing away the singularities of $C_3$ quotients of fake projective planes. One can try to emulate this construction to get such surfaces as complete intersections in homogeneous varieties.\\[.2em]

\item
In the opposite direction, we have constructed multiple fake projective planes with $C_3$ action. It is worthwhile to try to deform the FPP/3 surfaces $X_3$, $(d^2D)_3$, $(dD)_3$ and $d_3$ to get smooth surfaces with $K_S^2=3$ and $p_g=q=0$ with other fundamental groups. \\[.2em]

\item 
We have technically not proved that the fake projective planes we have constructed are non-isomorphic to each other. It is clear from our method, but some of the intermediate calculations are done with random points, which are less certain than symbolic computations. There does not appear to be a simple a posteriori calculation that would establish it, although computing various invariants may suffice. A natural approach of looking for linear changes of variables in $|2K|$ embedding would lead to equations in $99$ variables, which is far beyond what is currently feasible, even over a finite field.
\\[.2em]

\end{itemize}

\section{Appendix. Equations of $X$}\label{app}
Equations of the surface $X$ which is a double cover of $\PP^2_{fake}/C_3$ in its bicanonical embedding are the following. First, there are three even quadratic equations, where
we use the notation 
$\sigma
(W_0:\cdots:W_7)
=
(W_0:W_2:W_3:W_1:W_5:W_6:W_4:W_7)
$.
$$
\tiny
\begin{array}{rl}
eq_1=&(53321  - 119409 \ii \sqrt{15} )W_0 W_1 + 799064 W_1^2+( - 140437 + 
 134429 \ii \sqrt{15}) W_0 W_2 +( 10514 - 1103970 \ii \sqrt{15} )W_1 W_2 + 
 (87116 
 \\&
 - 15020 \ii \sqrt{15} )W_0 W_3 +(- 4478964 + 
 441588 \ii \sqrt{15} )W_1 W_3 +( 4468450 + 662382 \ii \sqrt{15} )W_2 W_3 - 
 799064 W_3^2 - 799064 W_4^2+
 \\&
 ( - 1461446 + 31542 \ii \sqrt{15}) W_4 W_5 
 +(-  94626  + 346962 \ii \sqrt{15} )W_4 W_6 + (1556072  - 
 378504 \ii \sqrt{15} )W_5 W_6 + 799064 W_6^2 + (3552360 
 \\&
  - 
 255192 \ii \sqrt{15}) W_4 W_7+( - 331128  + 358680 \ii \sqrt{15} )W_5 W_7
 +( -  3221232  - 103488 \ii \sqrt{15} )W_6 W_7\\[.2em] 
 eq_2=&\sigma(eq_1)\\[.2em]
    eq_3=&\sigma^2(eq_1)  \\[.2em]
 \end{array}
 $$
 Then there are two odd cubic equations (in addition to products of odd variables and even quadratic equations). These equations happen to be $C_3$-invariant.
 $$
 \tiny
\begin{array}{rl}
eq_4=&
( -98948224478443260  - 3058298456825380 \ii \sqrt{15}) W_0( W_1 W_4 +W_2 W_5 + W_3 W_6)
 +( - 
 1690952291170184385  + 
  \\&
 369304304128884495 \ii \sqrt{15} )W_4 W_5 W_6
 +
 (81462632898956280  - 23953315200421720 \ii \sqrt{15} )W_0 (W_2 W_4 + W_3 W_5 + W_1 W_6)
 \\&
 - 
 (1094836267717214400 + 
 398810644281337680 \ii \sqrt{15} )(W_1 W_2 W_4 +W_2 W_3 W_5 + W_1 W_3 W_6)
 +( - 1543563136889220  - 
 \\&
 50501498192304700 \ii \sqrt{15} )W_0 (W_3 W_4 +W_1 W_5 + W_2 W_6)+ ( - 
 1695126293066323845
 - 
 246880524310909965 \ii \sqrt{15}) (W_2 W_3 W_4
 \\& + W_1 W_3 W_5 + W_1 W_2 W_6)
 +(- 
 551526491689636680 + 326575488081046680 \ii \sqrt{15} )(W_3^2 W_4 + W_1^2 W_5 + W_2^2 W_6)
 + 
 \\&
 138080956729877280 (W_3^2 W_5 + W_1^2 W_6 + W_2^2 W_4)
 + 
 (1373637819762610920 + 
 199151284945609080 \ii \sqrt{15} )(W_2 W_3 W_6 + W_1 W_3 W_4 
 \\&+ W_1 W_2 W_5)
 +( - 5328591207840000  + 
2144900495488320 \ii \sqrt{15} W_0^2 W_7 )
 +(- 
 688957430544122040 
 \\&
  -
 158946397497914040 \ii \sqrt{15}) (W_1^2+W_2^2 +W_3^2) W_7 
 + (284310840392140950 + 
 106996549176297210 \ii \sqrt{15} )W_0 W_7( W_1 + W_2+ W_3 )
 \\&
 +( 
 1200158218487790765 
 - 
 641150643161964915 \ii \sqrt{15} )(W_1 W_2 +W_1 W_3 + W_2 W_3) W_7 
 + 
 (1808067319615295865
 \\&
 - 
 379632318487526223 \ii \sqrt{15} )(W_4 W_5 + W_4 W_6 +W_5 W_6)W_7  +( - 
 1999928222013559005  
+ 
 443220177629557755 \ii \sqrt{15} )(W_4 +W_5
 \\&
 +W_6)W_7^2
  +( 2471992417171938465 
 - 
 903652627464942327 \ii \sqrt{15}) W_7^3
 \\[.2em] 
 eq_5=&
  (4624568052886208 + 1836658084768192 \ii\sqrt{15}) W_0 (W_1 W_4 + W_2 W_5 + W_3 W_6)
 +(- 25723910738105944
 \\&
  - 4926352855196696 \ii\sqrt{15} )W_0 (W_2 W_4 + W_3 W_5 + W_1 W_6)
 +( 307036825631433576  + 57793461876095400 \ii\sqrt{15} )(W_1 W_2 W_4
 \\&
  +W_2 W_3 W_5+W_1 W_3 W_6)
 +(- 42812288376882136 - 3212847068911256 \ii\sqrt{15} )W_0 (W_3 W_4+ W_2 W_5 + W_1 W_6)
  +
  \\&
  ( 146468880652808448  -  67854829136903424 \ii\sqrt{15} )(W_1 W_3 W_4 +W_1 W_2 W_5+W_2 W_3 W_6)
  +(-  263825659127933973  
  \\&
  + 97123836640946787 \ii\sqrt{15} )(W_2 W_3 W_4 +W_1 W_3 W_5+W_1 W_2 W_6 )
 +( 244274695087257576  +  52031300557563432 \ii\sqrt{15} )(W_3^2 W_4 
 \\&
 +W_1^2 W_5 +  W_2^2 W_6)
 + 1753408974347648 W_0^2 (W_4 + W_5 + W_6)
 +  (413162544996145035 + 86812915194403587 \ii\sqrt{15} )W_4 W_5 W_6
 \\&
 +( - 2219261238917376  +  41371720329984 \ii\sqrt{15} )W_0^2 W_7 
  +( 112661548728436848 - 15940846449910032 \ii\sqrt{15} )W_0 W_7(W_1+W_2 +W_3) 
 \\&
 +(-  668268262244236563  -  103644995514656715 \ii\sqrt{15} )(W_1 W_2 + W_2 W_3 +W_1 W_3)W_7 
+( - 123783199950329640 
\\&
+ 32171290678282392 \ii\sqrt{15} )(W_1^2+W_2^2+W_3^2) W_7 
 +( -  574830176340661227- 133758869453715555 \ii\sqrt{15} )( W_4 W_5  + W_5 W_6 
 \\&
 + W_4 W_6)W_7 
 + ( 642303707483671947 W_4 W_7^2 + 134647945440329475 \ii\sqrt{15} (W_4+W_5+W_6) W_7^2 
+(-  1138685084396455995 
\\&
- 184628475341453619 \ii\sqrt{15}) W_7^3
  \end{array}
 $$
 Finally, there are four even cubic equations (again, in addition to products of even variables and degree two equations).
 $$
  \tiny
\begin{array}{rl}
 eq_6=&(-1666764770896080 + 24780099758160 \ii\sqrt{15}) W_0^3 + (-26655160103850225 - 
    141099466791655 \ii\sqrt{15}) W_0^2 W_1 + (68978091270130770 
  \\&
  - 
    43638746552374050 \ii\sqrt{15}) W_0 W_1^2 + (-14960300456911665 + 
    243600535544345 \ii\sqrt{15}) W_0^2 W_2 + (-136997854837433730 
     \\&
      + 
    100084814202953370 \ii\sqrt{15}) W_0 W_1 W_2 + (-2166930718788380580 - 
    114990006031685820 \ii\sqrt{15}) W_1^2 W_2 + (188045734378652430
      \\& - 
    11700924076935870 \ii\sqrt{15}) W_0 W_2^2 + (-334721551455006660 + 
    371044998805248180 \ii\sqrt{15}) W_1 W_2^2 + (-3471146893213005
      \\& + 
    3488453404965845 \ii\sqrt{15}) W_0^2 W_3 + (370759844061735240 + 
    142387946334768960 \ii\sqrt{15}) W_0 W_1 W_3 + (1178593751048740740
      \\& + 
    258386774372850060 \ii\sqrt{15}) W_1^2 W_3 + (337733550357073080 - 
    23905566679900920 \ii\sqrt{15}) W_0 W_2 W_3 + (3735489413393712765 
      \\&- 
    810504711260264475 \ii\sqrt{15}) W_1 W_2 W_3 + (-1479640375942060500 - 
    541366284526196220 \ii\sqrt{15}) W_2^2 W_3 + (88187928349336200 
      \\&+ 
    45648803952015264 \ii\sqrt{15}) W_0 W_4 W_5 + (-25167050881043400 + 
    54618239319319368 \ii\sqrt{15}) W_1 W_4 W_5 + (691467466176945360 
      \\&- 
    135405264642391296 \ii\sqrt{15}) W_2 W_4 W_5 + (731586446190101565 - 
    157420661770451835 \ii\sqrt{15}) W_3 W_4 W_5 + (-22789739259859320
      \\& + 
    14076296148587904 \ii\sqrt{15}) W_0 W_4 W_6 + (471919011020770320 - 
    286625927180712672 \ii\sqrt{15}) W_1 W_4 W_6 + (-52065650645825355 
      \\&- 
    336326621647458771 \ii\sqrt{15}) W_2 W_4 W_6 + (-71375582659697820 + 
    34183213308069084 \ii\sqrt{15}) W_0 W_5 W_6 + (544930562815031925 
      \\&- 
    262568409537346227 \ii\sqrt{15}) W_1 W_5 W_6 + (-773851724743350960 - 
    167052236438578176 \ii\sqrt{15}) W_2 W_5 W_6 + (-83927877818450400
      \\& - 
    17377009952576400 \ii\sqrt{15}) W_0 W_4 W_7 + (-36849898243929960 + 
    28649673281993400 \ii\sqrt{15}) W_1 W_4 W_7 + (-549479940520862295 
      \\&+ 
    243736669319493105 \ii\sqrt{15}) W_2 W_4 W_7 + (-796822141111243695 + 
    45136039971948345 \ii\sqrt{15}) W_3 W_4 W_7 + (53307405771262020 
      \\&
      - 
    21625406694472140 \ii\sqrt{15}) W_0 W_5 W_7 + (832604580620480745 + 
    48931912928647665 \ii\sqrt{15}) W_1 W_5 W_7 + 
 356264465750010720 W_2 W_5 W_7 
   \\&+ (-313200156644634105 + 
    130270143553698015 \ii\sqrt{15}) W_3 W_5 W_7 + (42156434307162420 + 
    3936635595601380 \ii\sqrt{15}) W_0 W_6 W_7 
      \\&+ (-1166733390875518935 + 
    266638938070257585 \ii\sqrt{15}) W_1 W_6 W_7 + (988209940857367095 + 
    425481801320871855 \ii\sqrt{15}) W_2 W_6 W_7 
      \\&+ (-142547635521932760 - 
    5281782721276824 \ii\sqrt{15}) W_0 W_7^2 + (232999883627690025 - 
    74107682752546863 \ii\sqrt{15}) W_1 W_7^2 
      \\&+ (-876915653107092165 - 
    319487180407733709 \ii\sqrt{15}) W_2 W_7^2 + (314943079584183435 - 
    119573778905614845 \ii\sqrt{15}) W_3 W_7^2
  \\[.2em]
   eq_7=&\sigma(eq_6)\\[.2em]
    eq_7=&\sigma^2(eq_6)\\[.2em]
   eq_9=&
   (324554939451 - 70795097523 \ii\sqrt{15}) W_0^3 
   + (1529614704078 -  1205210352894 \ii\sqrt{15}) W_0^2 (W_1 +W_2+W_3)
   + (-5415939170712 
     \\&
     + 1474501654360 \ii\sqrt{15}) W_0 (W_1^2+W_2^2+W_3^2)
    + (-62405246503404 - 10367134749412 \ii\sqrt{15}) W_0( W_1 W_2 +W_2 W_3 +W_1 W_3) 
   \\&   
 +  180800780856192 (W_1^2 W_2 + W_2^2 W_3 + W_3^2 W_1)
    + (-146202216215328 - 15520405118400 \ii\sqrt{15})( W_1 W_2^2 + W_2 W_3^2 + W_3 W_1^2) 
     \\&
     + (-341956310464440 + 182430423393624 \ii\sqrt{15}) W_1 W_2 W_3   
     + (-24797684639448 - 4795122965976 \ii\sqrt{15}) W_0 (W_4 W_5 + W_5 W_6
       \\& + W_4 W_6)
       + (-36568462213584 + 31673565214704 \ii\sqrt{15}) (W_2 W_4 W_5 + W_3 W_5 W_6 + W_1 W_6 W_4)
     + (-58201519194000 
       \\&
       + 40317236988336 \ii\sqrt{15}) (W_3 W_4 W_5 + W_1 W_5 W_6 + W_2 W_ 6 W_4)
+ (19496920450896 +     2134605859344 \ii\sqrt{15}) W_0 W_7 (W_4 + W_5 +W_6) 
  \\&
  + (21063286582128 - 6992925745680 \ii\sqrt{15}) (W_1 W_4 + W_2 W_5 + W_3 W_6)W_7 
+ (99385147864044 -  46684295204148 \ii\sqrt{15}) (W_2 W_4
  \\& + W_3 W_5 + W_1 W_6)W_7 
+ (1823973493356 -  15659763667380 \ii\sqrt{15}) (W_3 W_4+ W_1 W_5 +W_2 W_6) W_7   
  + (-95094214488 
    \\&-  2876821724952 \ii\sqrt{15}) W_0 W_7^2 
  + (-38259379316088 +  23825658096072 \ii\sqrt{15}) (W_1+W_2+W_3) W_7^2   
   \end{array}
$$


\begin{thebibliography}{8}

\bibitem[BCP11]{bauercatanese}
I.~Bauer, F.~Catanese, and R.~Pignatelli.
\newblock Surfaces of general type with geometric genus zero: a survey.
\newblock {\em Complex and differential geometry}. Springer, Berlin, Heidelberg, 2011. 1–48.

\bibitem[BK19]{BorKeum}L.~Borisov and J.~Keum,
\newblock Explicit equations of a fake projective plane
\newblock \emph{to appear in Duke Math. J.}, arXiv:1802.06333

\bibitem[BY18]{by}L.~Borisov and S.-K.~Yeung,
\newblock Explicit equations of the Cartwright-Steger surface,
\newblock arXiv:1804.00737.


\bibitem[BF19]{support}L.~Borisov and E.~Fatighenti,
\newblock New explicit constructions of surfaces of general type, supporting files.
\newblock https://www.math.rutgers.edu/\char`\~borisov/FPP-C14/ 


\bibitem[CS11]{cs}D.~Cartwright and T.~Steger,
\newblock Enumeration of the 50 fake projective planes.
\newblock {\em Math. Acad. Sci. Paris} 348 (2010), no. 1-2, 11–13.


\bibitem[CS11+]{CSlist} http://www.maths.usyd.edu.au/u/donaldc/fakeprojectiveplanes/registerofgps.txt

\bibitem[Co89]{aut}M.J~Cowen,
\newblock Automorphisms of Grassmannians.
\newblock {\em Proceedings of AMS} Volume 106, Number I, May 1989.


\bibitem[Fa18]{ef} E.~Fatighenti,
\newblock Surfaces of general type with $p_g=1$, q=0 and Grassmannians,
\newblock \emph{Math. Nachr. 293 (2020), no. 1 , Pages 88-100 }.

\bibitem[IK98]{Ishida}M.~Ishida and F.~Kato,
\newblock The strong rigidity theorem for non-Archimedean uniformization, 
\newblock \emph{The Tohoku Mathematical Journal, Second Series,} 50 (4): 537–555.


\bibitem[Ke06]{Keum06}J.~Keum, 
\newblock A fake projective plane with an order 7 automorphism, 
\newblock \emph{Topology. an International Journal of Mathematics,} 45 (5): 919–927.

\bibitem[Ke12]{Keum.comment}J.~Keum,
\newblock
Toward a geometric construction of fake projective planes,
\newblock \emph{Atti Accad. Naz. Lincei Rend. Lincei Mat. Appl.} 23 (2012), no. 2, 137–155. 

\bibitem[KK02]{KK}V.~Kulikov and V.~Kharlamov,
\newblock On real structures on rigid surfaces, 
\newblock \emph{Izv. Math.,} 66 (1), 133–150

\bibitem[Kl03]{Klingler}B.~Klingler, 
\newblock Sur la rigidit\'e de certains groupes fondamentaux, l'arithm\' eticit\'e des r\'eseaux hyperboliques complexes, et les faux plans projectifs, 
\newblock {\em Inventiones Mathematicae}, 153 (1), 105–143.

\bibitem[Mu79]{Mumford}D.~Mumford,
\newblock An algebraic surface with $K$ ample, $(K^2)=9$, $p_g=q=0$, 
\newblock \emph American Journal of Mathematics, 101 (1): 233–244.

\bibitem[PY07]{PY}G.~Prasad and S.-K.~Yeung, 
\newblock Fake projective planes,
\newblock \emph{Inventiones Mathematicae,} 168 (2): 321–370.


\bibitem[PY10]{PY2}G.~Prasad and S.-K.~Yeung, 
\newblock  Addendum to "Fake projective planes",
\newblock \emph{Inventiones Mathematicae,} 182 (1): 213–227.

\bibitem[Ri19]{Rito}  C.~Rito, 
\newblock
Surfaces with canonical map of maximum degree,
\newblock arXiv:1903.03017.

\bibitem[Y17]{Y17}S.-K.~Yeung, 
\newblock A surface of maximal canonical degree,
\newblock \emph{Math. Ann.,} 368 (2017), no. 3-4, 1171–1189. 

\bibitem[Y19]{Y19}S.-K.~Yeung, 
\newblock Erratum to [Y17],
\newblock https://www.math.purdue.edu/$\sim$yeungs/papers/cor.pdf

\end{thebibliography}
\end{document}